\documentclass[11pt]{article}
\usepackage{amssymb,amsmath,amsthm,verbatim,tikz,graphicx,fullpage}
\newtheorem{theorem}{Theorem}
\newtheorem{conjecture}{Conjecture}
\newtheorem{proposition}[theorem]{Proposition}
\newtheorem{corollary}[theorem]{Corollary}
\newtheorem{claim}{Claim}
\newtheorem{lemma}[theorem]{Lemma}
\theoremstyle{definition}
\newtheorem{fact}{Fact}
\newtheorem{definition}{Definition}

\newtheorem{remark}{Remark}
\newtheorem{example}{Example}
\newtheorem{observation}{Observation}
\DeclareMathOperator{\Prob}{Prob}
\DeclareMathOperator{\Sym}{Sym}
\DeclareMathOperator{\Aut}{Aut}
\DeclareMathOperator{\Orb}{Orb}
\DeclareMathOperator{\Stab}{Stab}
\DeclareMathOperator{\Aff}{Aff}
\DeclareMathOperator{\CT}{CT}
\DeclareMathOperator{\GL}{GL}
\DeclareMathOperator{\Cay}{Cay}

\DeclareMathOperator{\Link}{Link}
\DeclareMathOperator{\rank}{rank}
\newcommand{\Id}{\textrm{Id}}

\DeclareMathOperator{\Isom}{Isom}

\begin{document}

\title{On the structure of graphs which are locally indistinguishable from a lattice}
\author{Itai Benjamini\footnote{Department of Mathematics, Weizmann Institute of Science, Israel.}\,\, and David Ellis\footnote{School of Mathematical Sciences, Queen Mary, University of London, UK. Research supported in part by a Visiting Fellowship from the Weizmann Institute of Science.}}
\date{August 2016}
\maketitle

\begin{abstract}
For each integer $d \geq 3$, we obtain a characterisation of all graphs in which the ball of radius $3$ around each vertex is isomorphic to the ball of radius 3 in $\mathbb{L}^d$, the graph of the $d$-dimensional integer lattice. The finite, connected graphs with this property have a highly rigid, `global' algebraic structure; they can be viewed as quotient lattices of $\mathbb{L}^d$ in various compact $d$-dimensional orbifolds which arise from crystallographic groups. We give examples showing that `radius 3' cannot be replaced by `radius 2', and that `orbifold' cannot be replaced by `manifold'.

In the $d=2$ case, our methods yield new proofs of structure theorems of Thomassen \cite{thomassen-torus} and of M\'arquez, de Mier, Noy and Revuelta \cite{noy}, and also yield short, `algebraic' restatements of these theorems.

Our proofs use a mixture of techniques and results from combinatorics, geometry and group theory.
\vspace{0.2cm}

\end{abstract}

\section{Introduction}
Many results in Combinatorics concern the impact of `local' properties on `global' properties of combinatorial structures (e.g. graphs). A natural `local' condition to impose on a graph, is that it be regular. If $d \in \mathbb{N}$, a graph is said to be {\em $d$-regular} if all its vertices have degree $d$. Regular graphs have been extensively studied, and satisfy some rather strong `global' properties. For example, a well-known conjecture of Nash-Williams states that if $G$ is an $n$-vertex, $d$-regular graph with $d \geq \lfloor n/2\rfloor$, then $G$ can be decomposed into edge-disjoint Hamiltonian cycles and at most one perfect matching; this was recently proved for all sufficiently large $n$ by Csaba, K\"uhn, Lo, Osthus and Treglown \cite{cklot}. On the other hand, $d$-regular graphs are still quite `flexible' structures. For example, it was proved by Bollob\'as \cite{bollobas} and independently by McKay and Wormald \cite{wormald-mckay} that for any fixed integer $d \geq 3$, $G_d(n)$ has trivial automorphism group with high probability, i.e.
$$\Prob\{|\Aut(G_d(n))| = 1\} \to 1 \quad \textrm{as } n \to \infty.$$
This result was later extended to any $d \in \{3,4,\ldots,n-4\}$ by Kim, Sudakov and Vu \cite{ksv}, answering a question of Wormald.

It is natural to ask what happens to the global structure of a graph when we impose a `local' condition which is stronger than being $d$-regular. A natural condition to impose is that the subgraph induced by the ball of radius $r$ in $G$ around any vertex, is isomorphic to some fixed graph, for some fixed, small $r \in \mathbb{N}$. We proceed to give definitions which make this precise.

If $G$ is a (simple, undirected) graph, we write $V(G)$ for the vertex-set of $G$ and $E(G)$ for its edge-set. If $S \subset V(G)$, we write $G[S]$ for the subgraph of $G$ induced by $S$, i.e. the maximal subgraph of $G$ with vertex-set $S$. If $v,w \in V(G)$, the {\em distance from $v$ to $w$ in $G$} is defined to be the minimum number of edges in a path from $v$ to $w$ in $G$; it is denoted by $d_{G}(v,w)$. If $v \in V(G)$, and $r \in \mathbb{N}$, we define $\Link_r(v,G)$ to be the subgraph of $G$ induced by the set of vertices of $G$ with distance at most $r$ from $v$, and we define
$$\Link_r^{-}(v,G) := \Link_r(v,G) \setminus \{\{x,y\} \in E(G):\ d_G(v,x)=d_G(v,y)=r\}.$$
A {\em rooted graph} is an ordered pair $(G,v)$ where $G$ is a graph, and $v \in V(G)$.

Our key definitions are as follows.
\begin{definition}
\label{defn:r-locally}
If $(F,u)$ is a rooted graph, we say that a graph $G$ is {\em $r$-locally $(F,u)$} if for every vertex $v \in V(G)$, there exists a graph isomorphism $\phi: \Link_r(u,F) \to \Link_r(v,G)$ such that $\phi(u)=v$.
\end{definition}

\begin{definition}
\label{defn:weakly-r-locally} If $(F,u)$ is a rooted graph, we say that a graph $G$ is {\em weakly $r$-locally $(F,u)$} if for every vertex $v \in V(G)$, there exists a graph isomorphism $\phi: \Link_r^{-}(u,F) \to \Link_r^{-}(v,G)$ such that $\phi(u)=v$.
\end{definition}

Clearly, we have the implications
$$ G \textrm{ is } r\textrm{-locally }(F,u) \Rightarrow G \textrm{ is weakly } r\textrm{-locally }(F,u) \Rightarrow G \textrm{ is }(r-1)\textrm{-locally }(F,u),$$
for any $r \in \mathbb{N}$.

We remark that if $F$ is vertex-transitive, then Definitions \ref{defn:r-locally} and \ref{defn:weakly-r-locally} are independent of the choice of $u$. Hence, if $F$ is a vertex-transitive graph, we say that a graph $G$ is {\em $r$-locally $F$} if there exists $u \in V(F)$ such that $G$ is $r$-locally $(F,u)$. Similarly, we say that $G$ is {\em weakly $r$-locally $F$} if there exists $u \in V(F)$ such that $G$ is weakly $r$-locally $(F,u)$.

As a simple example, if $d \in \mathbb{N}$, let $T_d$ denote the infinite $d$-regular tree. A graph $G$ is $r$-locally $T_d$ if and only if it is a $d$-regular graph with girth at least $2r+2$, and is weakly $r$-locally $T_d$ if and only if it is a $d$-regular graph with girth at least $2r+1$.

Perhaps not surprisingly, for many rooted graphs $(F,u)$, graphs which are $r$-locally-$(F,u)$ (for small $r$) have a very rigid global structure. Sometimes, there is only one such connected graph up to isomorphism --- for example, the only connected graph which is 1-locally-$K_{t+1}$, is $K_{t+1}$ itself.

The $r=1$ case of Definition \ref{defn:r-locally} is familiar in the literature. If $G$ is a graph and $v \in V(G)$, we write $\Gamma(v)$ for the set of neighbours of $v$, and we write $L(v,G) : = G[\Gamma(v)]$ for the subgraph of $G$ induced by the neighbours of $v$; $L(v,G)$ is often called the {\em link of $G$ at $v$}. Note that a graph $G$ is $1$-locally $(F,u)$ if and only if $L(v,G) \cong L(u,F)$ for every vertex $v$ of $G$. (Graphs which are $1$-locally-$(F,u)$ for some rooted graph $(F,u)$ are usually called {\em graphs of constant link}.)

Many authors have given succinct necessary or sufficient conditions on $H$ for there to exist finite (or, in some cases, possibly infinite) graphs of constant link $H$, for graphs $H$ within various classes; results of this kind can be found e.g. in \cite{brown-connelly,bugata,doyen,hall-4,sedlacek,weetman-1}. However, the problem seems very hard in general. For more background, the reader may consult the survey \cite{hell}.

In this paper, we will focus on the case where $F$ is a Euclidean lattice. If $d \in \mathbb{N}$, the {\em $d$-dimensional lattice} $\mathbb{L}^d$ is the graph with vertex-set $\mathbb{Z}^d$, and edge-set
$$\{\{x,x+e_i\}:\ x \in \mathbb{Z}^d, i \in [d]\},$$
where $e_i = (0,0,\ldots,0,1,0,\ldots,0)$ denotes the $i$th unit vector in $\mathbb{R}^d$. We study the properties of graphs which are $r$-locally $\mathbb{L}^d$ or weakly $r$-locally $\mathbb{L}^d$, for various $r$. We feel $\mathbb{L}^d$ is a natural `second' case to study, the `first' case perhaps being $T_{2d}$. Note that $T_{2d}$ is the standard Cayley graph of the free group on $d$ generators, and as remarked above, a graph is $r$-locally $T_{2d}$ if and only if it is a $(2d)$-regular graph of girth at least $2r+2$; regular graphs of high girth have been intensively studied (see e.g. \cite{hoory-thesis,luw,lps}). By comparison, $\mathbb{L}^d$ is the standard Cayley graph of the free {\em Abelian} group on $d$ generators.

Note that a graph is $r$-locally $\mathbb{L}^1$ if and only if it is a vertex-disjoint union of cycles of length at least $2r+2$, and is weakly $r$-locally $\mathbb{L}^1$ if and only if it is a vertex-disjoint union of cycles of length at least $2r+1$, so the first interesting case is $d=2$. 

It turns out that for all $d \geq 2$ and all $r \geq 3$, graphs which are weakly $r$-locally $\mathbb{L}^d$ have a very rigidly proscribed, algebraic global structure, in stark contrast to regular graphs of high girth. By comparison, the uniform random $d$-regular graph $G_d(n)$ can be generated using a simple, purely combinatorial process, namely, the Configuration Model of Bollob\'as \cite{config}, and $G_d(n)$ has girth at least $g$ with positive probability for any fixed $d,g \geq 3$ (see \cite{bollobas-cycles,wormald-thesis}). 

The highly rigid, algebraic structure of graphs which are $r$-locally-$\mathbb{L}^d$ is not perhaps very surprising, in the light of results on complexes with specified links. Indeed, one may associate to a graph $G$, a 2-dimensional CW complex $S$, by attaching 2-cells to certain of its cycles; the property of $G$ being weakly $r$-locally-$(F,u)$ then translates to a property of the links or the combinatorial $r$-balls of $S$. (Rigidity phenomena for complexes with specified links or combinatorial $r$-balls have been observed by many authors, e.g. by Tits \cite{tits-1,tits-2} in his celebrated work on buildings, and more recently e.g. by Ballman and Brin \cite{ballman-brin}, \'Swiatkowski \cite{swiatkowski} and Wise \cite{wise}.) One of our proofs takes advantage of this connection.

Our main result is the following structure theorem for graphs which are weakly $3$-locally $\mathbb{L}^d$.

\begin{theorem}
\label{thm:cover-d}
Let $d \geq 2$ be an integer, and let $G$ be a connected graph which is weakly $3$-locally $\mathbb{L}^d$. Then there exists a normal covering map\footnote{See section 2 for the definition of a normal covering map.} from $\mathbb{L}^d$ to $G$. 
\end{theorem}

We will give a construction (Example \ref{example:Ld}) to show that Theorem \ref{thm:cover-d} is best possible for each $d \geq 3$, in the sense that for each $d \geq 3$, there exist finite connected graphs which are 2-locally $\mathbb{L}^d$ but which are not covered by $\mathbb{L}^d$. The construction is algebraic.

In the $d=2$ case, we prove something slightly stronger, yielding new proofs of structure theorems of Thomassen \cite{thomassen-torus} and of M\'arquez, de Mier, Noy and Revuelta \cite{noy}, and also yielding short, `algebraic' restatements of these theorems. To state our result, we need the following definition, implicit in \cite{thomassen-torus}.

\begin{definition}
Let $G$ be a graph. We say that $G$ has the {\em 4-cycle wheel property} if it is 4-regular, and there exists a family $\mathcal{C}$ of 4-cycles of $G$ such that for every vertex $v$ of $G$, there are exactly four 4-cycles in $\mathcal{C}$ which contain $v$, and the union of these four 4-cycles is isomorphic to the following graph.
\end{definition}
\begin{center}
\begin{tikzpicture}
\tikzstyle{every node}=[draw,circle,fill=black,minimum size=6pt,inner sep=0pt]
\draw (0,0) node {};
\draw (0,1) node {};
\draw (0,-1) node {};
\draw (-1,0) node {};
\draw (1,0) node {};
\draw (1,1) node {};
\draw(-1,1) node {};
\draw (-1,-1) node {};
\draw (1,-1) node {};
\draw (-1,0) -- (0,0) -- (1,0);
\draw (0,-1) -- (0,0) -- (0,1);
\draw (-1,0) -- (-1,1) -- (0,1) -- (1,1) -- (1,0) -- (1,-1) -- (0,-1) -- (-1,-1) -- (-1,0);
\end{tikzpicture}
\end{center}

Note that a graph which is weakly 2-locally-$\mathbb{L}^2$ has the 4-cycle wheel property; on the other hand, a $3 \times 3$ discrete torus has 4-cycle wheel property, but is not weakly $2$-locally $\mathbb{L}^2$.

We prove the following structure theorem for graphs with the 4-cycle wheel property.

\begin{theorem}
\label{thm:cover-2}
Let $G$ be a connected graph with the 4-cycle wheel property. Then there exists a normal covering map from $\mathbb{L}^2$ to $G$.
\end{theorem}

Our (short) proof of Theorem \ref{thm:cover-2} takes advantage of some classical techniques in geometry. Namely, we use the family $\mathcal{C}$ of 4-cycles to define a 2-dimensional cubical CW complex $S$ in the natural way, and we consider the universal cover of $S$.

We note that the technique of `promoting' a graph to a 2-complex using its cycles has been used before in the context of graphs with constant link, for example by Nedela in \cite{nedela,nedela-2}. We note also that Weetman \cite{weetman-1} obtained results on graphs of constant link by considering topological properties of the {\em clique complex} of a graph (the simplicial complex whose simplices correspond to cliques in the graph).

We remark that Theorem \ref{thm:cover-d} can also be given a proof along similar lines, by using the 4-cycles of $G$ to define a 2-dimensional cubical CW complex $S$ in the natural way, and considering an appropriate `thickening' of $S$ to an incomplete flat manifold. However, as this proof appeals to geometric machinery which may unfamiliar to some readers, we have chosen to present here a somewhat longer, but more elementary and more combinatorial proof of Theorem \ref{thm:cover-d}.

We use some fairly standard arguments from topological graph theory and group theory to deduce the following from Theorem \ref{thm:cover-2}.

\begin{corollary}
\label{corr:torusklein2}
Let $G$ be a finite, connected graph with the 4-cycle wheel property. Then there exists a subgroup $\Gamma \leq \Aut(\mathbb{L}^2)$ which acts freely on $\mathbb{L}^2$, such that $G$ is isomorphic to the quotient graph\footnote{See Definitions \ref{definition:free-action} and \ref{definition:quotient} in section 2 for the definitions of a quotient graph and a free action.} $\mathbb{L}^2/\Gamma$, and such that the orbit space $\mathbb{R}^2/\Gamma$ is either a torus or a Klein bottle.
\end{corollary}
It follows that we can view any finite, connected graph with the 4-cycle wheel property as a `quotient lattice' of $\mathbb{L}^2$ inside a torus or a Klein bottle; in particular, it must be a quadrangulation of the torus or of the Klein bottle, as observed by Thomassen in \cite{thomassen-torus}. One can use Corollary \ref{corr:torusklein2} to deduce the structure theorems for graphs with the 4-cycle wheel property which are proved by Thomassen \cite{thomassen-torus} and by M\'arquez, de Mier, Noy and Revuelta \cite{noy}, and also to give short, clean, `algebraic' descriptions of the graphs which are defined combinatorially (at some length) in these structure theorems. We omit the details of these deductions, as they are straightforward.

Likewise, in the $d \geq 3$ case, we deduce the following from Theorem \ref{thm:cover-d}.

 \begin{corollary}
 \label{corr:structure-d}
 Let $d \geq 3$ be an integer, and let $G$ be a finite, connected graph which is weakly $3$-locally $\mathbb{L}^d$. Then there exists a subgroup $\Gamma \leq \Aut(\mathbb{L}^d)$ which acts freely on $\mathbb{L}^d$, such that $G$ is isomorphic to the quotient graph $\mathbb{L}^d/\Gamma$. Moreover, viewed as a subgroup of $\Isom(\mathbb{R}^d)$, $\Gamma$ is a $d$-dimensional crystallographic group, and the orbit space $\mathbb{R}^d/\Gamma$ is a compact topological orbifold.
 \end{corollary}
It follows that we can view any finite, connected graph which is weakly $3$-locally $\mathbb{L}^d$, as a `quotient lattice' of $\mathbb{L}^d$ inside a compact $d$-dimensional topological orbifold. (See Definition \ref{defn:orbifold} for the definition of a topological orbifold.) We will give an example to show that for each $d \geq 7$, the orbit space $\mathbb{R}^d/\Gamma$ in Corollary \ref{corr:structure-d} need not be a topological manifold (see Example \ref{example:not-manifold}).

\begin{remark}
Bieberbach's theorems \cite{bieberbach-1911,bieberbach-1912} imply that for any $d \in \mathbb{N}$, there are only a finite number ($f(d)$, say) of affine-conjugacy classes of $d$-dimensional crystallographic groups (where two crystallographic groups are said to be {\em affine-conjugate} if they are conjugate via an affine transformation of $\mathbb{R}^d$). It follows that the orbit space $\mathbb{R}^d/\Gamma$ in Corollary \ref{corr:structure-d} is homeomorphic to one of at most $f(d)$ topological spaces. (See Fact \ref{fact:bieberbach}.)
\end{remark}

We also obtain an exact (algebraic) characterisation of the graphs which are $3$-locally $\mathbb{L}^d$, or weakly $3$-locally $\mathbb{L}^d$, for each $d \geq 3$ (Proposition \ref{prop:quotient-d}). As its statement is a little more technical than that of Corollary \ref{corr:structure-d}, we defer it until section 4.
\newline

The remainder of this paper is structured as follows. In section 2, we give the definitions, background and standard tools which we require from topological graph theory, topology and group theory. This section is rather long, but as various `standard' texts in topology and geometry use slightly different conventions, we prefer to set out our conventions in full, to avoid ambiguity. All or part of section 2 can be skipped by readers familiar with the relevant areas. In section 3, we prove Theorems \ref{thm:cover-d} and \ref{thm:cover-2}, and we construct an example showing that Theorem \ref{thm:cover-d} is best possible in a certain sense. In section 4, we deduce Corollary \ref{corr:torusklein2} from Theorem \ref{thm:cover-2} and Corollary \ref{corr:structure-d} from Theorem \ref{thm:cover-d}, using some fairly standard arguments from topological graph theory and group theory. We conclude with a discussion of some open problems and related results in section \ref{sec:conc}.

\section{Definitions, background and tools}

\subsubsection*{Basic graph-theoretic notation and terminology}

Unless otherwise stated, all graphs will be undirected and simple (that is, without loops or multiple edges); they need not be finite. An undirected, simple graph is defined to be an ordered pair of sets $(V,E)$, where $E \subset {V \choose 2}$; $V$ is called the {\em vertex-set} and $E$ the {\em edge-set}. An edge $\{v,w\}$ of a graph will often be written $vw$, for brevity.

If $S \subset V(G)$, we let $N(S)$ denote the {\em neighbourhood} of $S$, i.e.
$$N(S) = S \cup \{v \in V(G):\ sv \in E(G)\textrm{ for some } s \in S\}.$$
We say a graph is {\em locally finite} if each of its vertices has only finitely many neighbours.

If $G$ is a graph, and $u,v \in V(G)$ are in the same component of $G$, the {\em distance from $u$ to $v$ in $G$} is the minimum length of a path from $u$ to $v$; it is denoted by $d_{G}(u,v)$ (or by $d(u,v)$, when the graph $G$ is understood). A path of minimum length between $u$ and $v$ (i.e., a path of length $d_{G}(u,v)$) is called a {\em geodesic}. If $G$ is a graph, and $v$ is a vertex of $G$, the {\em ball of radius $r$ around $v$} is defined by
$$B_r(v,G) := \{w \in V(G):\ d_{G}(v,w) \leq r\},$$
i.e. it is the set of vertices of $G$ of distance at most $r$ from $v$.

\subsubsection*{Background and tools from topological graph theory}

We follow \cite{gross-tucker,nedela-survey}.

\begin{definition}
If $F$ and $G$ are graphs, and $p:V(F) \to V(G)$ is a graph homomorphism from $F$ to $G$, we say that $p$ is a {\em covering map} if $p$ maps $\Gamma(v)$ bijectively onto $\Gamma(p(v))$, for all $v \in V(F)$. In this case, we say that $F$ {\em covers} $G$.
\end{definition}

\begin{remark}
\label{remark:surjective}
It is easy to see that if $F$ and $G$ are graphs with $G$ connected, and $p:V(F) \to V(G)$ is a covering map, then $p$ is surjective.
\end{remark}

\begin{definition}
Let $F$ and $G$ be graphs, and let $p:V(F) \to V(G)$ be a covering of $G$ by $F$. The pre-image of a vertex of $G$ under $p$ is called a {\em fibre} of $p$.
\end{definition}

\begin{definition}
Let $F$ and $G$ be graphs, and let $p:V(F) \to V(G)$ be a covering of $G$ by $F$. An automorphism $\phi \in \Aut(F)$ is said to be a {\em covering transformation of $p$} if $p \circ \phi = p$. The group of covering transformations of $p$ is denoted by $\CT(p)$.
\end{definition}

Note that any covering transformation of $p$ acts on each fibre of $p$, but it need not act transitively on any fibre of $p$.

\begin{definition}
Let $F$ and $G$ be graphs, and let $p:V(F) \to V(G)$ be a covering of $G$ by $F$. We say that $p$ is a {\em normal} covering if $\CT(p)$ acts transitively on each fibre of $p$.
\end{definition}

\begin{remark}
It is well known (and easy to check) that if $F$ is a connected graph, and $p:V(F) \to V(G)$ is a covering of $G$ by $F$, then if $\CT(p)$ acts transitively on some fibre of $p$, it acts transitively on every fibre of $p$. Hence, in the previous definition, `on each fibre' may be replaced by `on some fibre'.
\end{remark}

\begin{definition}
Let $\Gamma$ be a group, let $X$ be a set, and let $\alpha:\Gamma \times X \to X$ be an action of $\Gamma$ on $X$. For each $x \in X$, we write $\Orb_{\Gamma}(x): = \{\alpha(\gamma,x):\ \gamma \in \Gamma\}$ for the $\Gamma$-orbit of $x$, and $\Stab_{\Gamma}(x) := \{\gamma \in \Gamma:\ \alpha(\gamma,x)=x\}$ for the stabiliser of $x$ in $\Gamma$. When the group $\Gamma$ is understood, we suppress the subscript $\Gamma$.
\end{definition}

\begin{definition}
If $F$ is a graph and $\Gamma \leq \Aut(F)$, the {\em minimum displacement} of $\Gamma$ is defined to be $D(\Gamma) : = \min\{d_F(x,\gamma(x)):\ x \in V(F),\ \gamma \in \Gamma \setminus \{\Id\}\}$.
\end{definition}

\begin{definition}
\label{definition:free-action}
Let $\Gamma$ be a group, let $X$ be a set, and let $\alpha:\Gamma \times X \to X$ be an action of $\Gamma$ on $X$. We say that $\alpha$ is {\em free} if $\alpha(\gamma,x) \neq x$ for all $x \in X$ and all $\gamma \in \Gamma \setminus \{\textrm{Id}\}$.
\end{definition}

\begin{definition}(Quotient of a graph.)
\label{definition:quotient}

Let $F$ be a simple graph, and let $\Gamma \leq \Aut(F)$. Then $\Gamma$ acts on $V(F)$ via the natural left action $(\gamma,x) \mapsto \gamma(x)$, and on $E(F)$ via the natural (induced) action $(\gamma,\{x,y\}) \mapsto \{\gamma(x),\gamma(y)\}$. We define the {\em quotient graph} $F/\Gamma$ to be the multigraph whose vertices are the $\Gamma$-orbits of $V(F)$, and whose edges are the $\Gamma$-orbits of $E(F)$, where for any edge $\{x,y\} \in E(F)$, the edge $\Orb(\{x,y\})$ has endpoints $\Orb(x)$ and $\Orb(y)$. Note that $F/\Gamma$ may have loops (if $\{x,\gamma(x)\} \in E(F)$ for some $\gamma \in \Gamma$ and some $v \in V(F)$), and it may also have multiple edges (if there exist $\{u_1,u_2\},\{v_1,v_2\} \in E(F)$ with $\gamma_1(u_1) = v_1$ and $\gamma_2(u_2) = v_2$ for some $\gamma_1,\gamma_2 \in \Gamma$, but $\{\gamma(u_1),\gamma(u_2)\} \neq \{v_1,v_2\}$ for all $\gamma \in \Gamma$).
\end{definition}

\begin{definition}
Let $G$ be a graph, and let $\Gamma \leq \Aut(G)$. We say that $\Gamma$ {\em acts freely on $G$} if the natural actions of $\Gamma$ on $V(G)$ and $E(G)$ are both free actions, or equivalently, if no element of $\Gamma \setminus \{\Id\}$ fixes any vertex or edge of $G$. 
\end{definition}

The following two lemmas are well-known, and easy to check.

\begin{lemma}
\label{lemma:free}
Let $F$ be a connected (possibly infinite) graph, let $G$ be a graph, and let $p:V(F) \to V(G)$ be a covering map from $F$ to $G$. Then $\CT(p)$ acts freely on $F$.
\end{lemma}

\begin{lemma}
\label{lemma:iso}
Let $F$ and $G$ be (possibly infinite) graphs with $G$ connected, and suppose $p:V(F) \to V(G)$ is a normal covering map from $F$ to $G$. Then there is a graph isomorphism between $G$ and $F/\CT(p)$.
\end{lemma}

\subsubsection*{Some background from topology and group theory}

\begin{fact}
The group $\Isom(\mathbb{R}^d)$ of isometries of $d$-dimensional Euclidean space satisfies
\begin{align*} \Isom(\mathbb{R}^d) & = \{t \circ \sigma:\ t \in T(\mathbb{R}^d),\ \sigma \in O(d)\}\\
& = \{\sigma \circ t:\ t \in T(\mathbb{R}^d),\ \sigma \in O(d)\}\\
& = T(\mathbb{R}^d) \rtimes O(d),
\end{align*}
where
$$T(\mathbb{R}^d) := \{x \mapsto x+v:\ v \in \mathbb{R}^d\}$$
denotes the group of all translations in $\mathbb{R}^d$, and $O(d) \leq \GL(\mathbb{R}^d)$ denotes the group of all real orthogonal $d \times d$ matrices.
\end{fact}

\begin{fact}
\label{fact:aut-ld}
For any $d \in \mathbb{N}$, we have
\begin{align*} \Aut(\mathbb{L}^d) & = \{t \circ \sigma:\ t \in T(\mathbb{Z}^d),\ \sigma \in B_d\}\\
& = \{\sigma \circ t:\ t \in T(\mathbb{Z}^d),\ \sigma \in B_d\}\\\
&= T(\mathbb{Z}^d) \rtimes B_d,
\end{align*}
where
$$T(\mathbb{Z}^d) :=  \{x \mapsto x+v:\ v \in \mathbb{Z}^d\}$$
denotes the group of all translations by elements of $\mathbb{Z}^d$, and
$$B_d = \{\sigma \in GL(\mathbb{R}^d):\ \sigma(\{\pm e_i:\ i \in [d]\}) = \{\pm e_i: \ i \in [d]\}\},$$
denotes the {\em $d$-dimensional hyperoctahedral group}, which is the symmetry group of the $d$-dimensional (solid) cube with set of vertices $\{-1,1\}^d$, and can be identified with the permutation group
$$\{\sigma \in \Sym([d] \cup \{-i:\ i \in [d]\}):\ \sigma(-i) = -\sigma(i)\ \forall i\},$$
in the natural way (identifying $e_i$ with $i$ and $-e_i$ with $-i$ for all $i \in [d]$). We therefore have $|B_d| = 2^d d!$.
\end{fact}

\begin{fact}
\label{fact:embedding-isometries}
It is easy to see that every element of $\Aut(\mathbb{L}^d)$ can be uniquely extended to an element of $\Isom(\mathbb{R}^d)$. We can therefore view $\Aut(\mathbb{L}^d)$ as a subgroup of $\Isom(\mathbb{R}^d)$.
\end{fact}

\begin{definition}
If $X$ is a topological space, and $\Gamma$ is a group acting on $X$, the {\em orbit space} $X/ \Gamma$ is the (topological) quotient space $X/ \sim$, where $x \sim y$ iff $y \in \Orb_{\Gamma}(x)$, i.e. iff $x$ and $y$ are in the same $\Gamma$-orbit.
\end{definition}

\begin{definition}
If $X$ is a topological space, a group $\Gamma$ of homeomorphisms of $X$ is said to be {\em discrete} if the relative topology on $\Gamma$ (induced by the compact open topology on the group of all homeomorphisms of $X$) is the discrete topology.
\end{definition}

\begin{definition}
If $X$ is a topological space, and $\Gamma$ is a discrete group of homeomorphisms of $X$, we say that $\Gamma$ acts {\em properly discontinuously} on $X$ if for any $x,y \in X$, there exist open neighbourhoods $U$ of $x$ and $V$ of $y$ such that $|\{\gamma \in \Gamma:\ \gamma(U) \cap V \neq \emptyset\}| < \infty$.
\end{definition}

\begin{fact}
\label{fact:discrete}
If $\Gamma \leq \Isom(\mathbb{R}^d)$, then $\Gamma$ is discrete if and only if for any $x \in \mathbb{R}^d$, the orbit $\{\gamma(x):\ \gamma \in \Gamma\}$ is a discrete subset of $\mathbb{R}^d$. Hence, $\Aut(\mathbb{L}^d)$ is a discrete subgroup of $\Isom(\mathbb{R}^d)$.
\end{fact}

\begin{fact}
\label{fact:proper-disc}
If $\Gamma \leq \Isom(\mathbb{R}^d)$ is discrete, then $\Gamma$ acts properly discontinuously on $\mathbb{R}^d$. (Note that it is clear directly from the definition that $\Aut(\mathbb{L}^d)$, and any subgroup thereof, acts properly discontinuously on $\mathbb{R}^d$.)
\end{fact}

\begin{definition}
Let $\Gamma$ be a discrete subgroup of $\Isom(\mathbb{R}^d)$. The {\em translation subgroup} $T_{\Gamma}$ of $\Gamma$ is the subgroup of all translations in $\Gamma$. The {\em lattice of translations} of $\Gamma$ is the lattice $L_{\Gamma} := \{\gamma(0):\ \gamma \in T_{\Gamma}\} \subset \mathbb{R}^d$. We have $L_{\Gamma} \cong \mathbb{Z}^r$ for some $r \in \{0,1,\ldots,d\}$; the integer $r$ is called the {\em rank} of the lattice $L_{\Gamma}$.
\end{definition}

\begin{definition}
\label{definition:crystallographic}
A discrete subgroup $\Gamma \leq \Isom(\mathbb{R}^d)$ is said to be a {\em $d$-dimensional crystallographic group} if its lattice of translations has rank $d$ (or, equivalently, if the orbit space $\mathbb{R}^d/\Gamma$ is compact).
\end{definition}

\begin{definition}
A group $\Gamma$ is said to be {\em torsion-free} if the only element of finite order in $\Gamma$ is the identity.
\end{definition}

\begin{definition}(Following \cite{charlap}.)
\label{definition:bieberbach}
A torsion-free $d$-dimensional crystallographic group is called a {\em $d$-dimensional Bieberbach group}.
\end{definition}

\begin{definition}
If $\Gamma$ is a $d$-dimensional crystallographic group, its {\em point group} $P_{\Gamma}$ is defined by
 $$P_{\Gamma} = \{\sigma \in O(d):\ t \circ \sigma \in \Gamma \textrm{ for some }t \in T(\mathbb{R}^d)\}.$$
 \end{definition}

\begin{fact} (Bieberbach's Theorems.)
\label{fact:bieberbach}
Bieberbach's First Theorem \cite{bieberbach-1911} implies that if $\Gamma$ is a $d$-dimensional crystallographic group, then $P_{\Gamma}$ is finite. It is easy to check that if $\Gamma$ is a $d$-dimensional crystallographic group, then its point group $P_{\Gamma}$ acts faithfully on the lattice $L_{\Gamma}$. Hence, $P_{\Gamma}$ is isomorphic to a finite subgroup of $GL_d(\mathbb{Z})$. The Jordan-Zassenhaus theorem implies that for any (fixed) $d \in \mathbb{N}$, there are only finitely many isomorphism-classes of finite subgroups of $\GL_d(\mathbb{Z})$. It follows that for any $d \in \mathbb{N}$, there are only finitely many possibilities for the isomorphism class of the point-group of a $d$-dimensional crystallographic group. Bieberbach's Third Theorem \cite{bieberbach-1912} says more: for any $d \in \mathbb{N}$, there are only finitely many possibilities for the isomorphism class\footnote{Meaning, as usual, an isomorphism class of abstract groups.} of a $d$-dimensional crystallographic group. Bieberbach's Second Theorem \cite{bieberbach-1912} states that if $\Gamma$ and $\Gamma'$ are $d$-dimensional crystallographic groups, and $\Phi:\Gamma \to \Gamma'$ is an isomorphism, then there exists an affine transformation $\alpha$ of $\mathbb{R}^d$ such that $\Phi(\gamma) = \alpha^{-1} \gamma \alpha$ for all $\gamma \in \Gamma$. Hence, two $d$-dimensional crystallographic groups are isomorphic if and only if they are conjugate in $\Aff(\mathbb{R}^d)$, the group of all affine transformations of $\mathbb{R}^d$. It follows that for any fixed $d \in \mathbb{N}$, there are only a finite number ($f(d)$, say) of possibilities for the affine-conjugacy class of a $d$-dimensional crystallographic group.
\end{fact}

\begin{definition}($d$-dimensional topological manifold.)

Let $d \in \mathbb{N}$. A Hausdorff topological space $X$ is said to be a {\em $d$-dimensional topological manifold} if for every $x \in X$, there exists an open neighbourhood of $x$ which is homeomorphic to an open subset of $\mathbb{R}^d$. (Note that we do not regard a `manifold with boundary' as a manifold.)
\end{definition}

\begin{definition}($d$-dimensional topological orbifold, following \cite{lectures-on-orbifolds}.)
\label{defn:orbifold}
\label{page:orbifold-definition}

Let $d \in \mathbb{N}$ be fixed. Let $X$ be a Hausdorff topological space. An {\em orbifold chart} on $X$ is a 4-tuple $(V,G,U,\pi)$, where
\begin{itemize}
\item $V$ is an open subset of $\mathbb{R}^d$;
\item $U$ is an open subset of $X$;
\item $G$ is a finite group of homeomorphisms of $V$;
\item $\pi = \phi \circ q$, where $q: V \to V/G$ is the orbit map (i.e. the map taking $v \in V$ to its orbit), and $\phi:V/G \to U$ is a homeomorphism.
\end{itemize}
We say that two orbifold charts $(V_1,G_1,U_1,\pi_1),\ (V_2,G_2,U_2,\pi_2)$ on $X$ are {\em compatible} if for any $v_1 \in V_1,\ v_2 \in V_2$ with $\pi_1(v_1) = \pi_2(v_2)$, there exist open neighbourhoods $W_i$ of $v_i$ in $V_i$ (for $i=1,2$), and a homeomorphism $h:W_2 \to W_1$, such that $\pi_2 | W_2 = (\pi_1 |W_1) \circ h$.

An {\em atlas of orbifold charts on $X$} is a collection $\{(V_i,G_i,U_i,\pi_i):\ i \in I\}$ of pairwise compatible orbifold charts on $X$ such that $\{U_i:\ i \in I\}$ is a cover of $X$. A {\em $d$-dimensional topological orbifold} is a pair $(X,\mathcal{A})$, where $X$ is a topological space, and $\mathcal{A}$ is an atlas of orbifold charts on $X$. (Abusing terminology slightly, if $(X,\mathcal{A})$ is a topological orbifold, we will sometimes refer to the underlying topological space $X$ as a topological orbifold.)
\end{definition}

Note that, for simplicity, all orbifolds (and manifolds) in this paper will be viewed only as topological ones. An `orbifold' (resp. `manifold') will therefore always mean a topological orbifold (resp. manifold).

\begin{remark}
Informally, a $d$-dimensional topological orbifold is a topological space, together with a collection of charts which model it locally using quotients of open subsets of $\mathbb{R}^d$ under the actions of finite groups (rather than simply using open subsets of $\mathbb{R}^d$, as in the definition of a manifold). Note that a $d$-dimensional manifold is precisely a $d$-dimensional orbifold where we can take each finite group $G_i$ in Definition \ref{defn:orbifold} to be the trivial group.
\end{remark}

\begin{example}
The orbit space $\mathbb{R}^2/\langle (x_1,x_2) \mapsto (x_1,-x_2) \rangle$ (which is homeomorphic to the closed upper half-plane) can be given the structure of a 2-dimensional orbifold. More generally, any `manifold with boundary' can be given the structure of an orbifold. Even more generally, if $M$ is a manifold, and $\Gamma$ is a finite group of homeomorphisms of $M$, then the orbit space $M/\Gamma$ can be given the structure of an orbifold.
\end{example}

\begin{fact}
\label{fact:proper-orbifold}(See \cite{orbifolds}, Proposition 13.2.1.)
If $\Gamma$ is a discrete group of homeomorphisms of a $d$-dimensional topological manifold $M$, and $\Gamma$ acts properly discontinuously on $M$, then the orbit space $M/\Gamma$ can be given the structure of a $d$-dimensional topological orbifold.
Indeed, for each $x \in M$, we can take a chart at $x$ where the finite group of homeomorphisms $G_x$ is isomorphic to $\Stab_{\Gamma}(x)$. If, in addition, $\Gamma$ is torsion-free, then $\Stab_{\Gamma}(x)$ is trivial for all $x$, so $M/\Gamma$ is a $d$-dimensional topological manifold.
\end{fact}

\begin{fact}
It follows from Facts \ref{fact:proper-disc} and \ref{fact:proper-orbifold} that if $\Gamma \leq \Isom(\mathbb{R}^d)$ is discrete, then $\mathbb{R}^d/\Gamma$ is a $d$-dimensional orbifold, and if in addition, $\Gamma$ is torsion-free, then $\mathbb{R}^d/\Gamma$ is a $d$-dimensional manifold. Therefore, if $\Gamma$ is a $d$-dimensional crystallographic group, then $\mathbb{R}^d/\Gamma$ is a compact $d$-dimensional orbifold, and if $\Gamma$ is a $d$-dimensional Bieberbach group, then $\mathbb{R}^d/\Gamma$ is a compact $d$-dimensional manifold. Since there are at most $f(d)$ possibilities for the affine-conjugacy class of a $d$-dimensional crystallographic group, the orbit space $\mathbb{R}^d/\Gamma$ is homeomorphic to one of at most $f(d)$ topological spaces.
\end{fact}

\section{Proofs of `topological' structure theorems}
\begin{proof}[Proof of Theorem \ref{thm:cover-2}.]
Let $G$ be as in the statement of the theorem. Let $\mathcal{C}$ be a collection of 4-cycles of $G$ as in the definition of the 4-cycle wheel property. There is an obvious way to produce from $G$ a 2-dimensional cubical CW complex $S$, by attaching a unit square to each of the 4-cycles in $\mathcal{C}$ (and gluing these unit squares together along edges of $G$); these unit squares are the 2-cells of the complex, and the edges of $G$ are the 1-cells. Since any edge of $G$ is contained in exactly two of the 4-cycles in $\mathcal{C}$, $S$ is a complete surface, and an isomorphic copy $G'$ of $G$ is embedded in $S$, the faces of $G'$ being precisely the 2-cells of $S$.

Clearly, the surface $S$ is flat (i.e., locally isometric to $\mathbb{R}^2$); since $S$ is a complete, flat surface, its universal cover is $\mathbb{R}^2$. Let $p: \mathbb{R}^2 \to S$ be a universal covering map. Note that $p$ is a normal covering map, since it is universal. Let $H = p^{-1}(G')$ denote the pull-back of $G'$. It is well-known (and easy to see, e.g. by composing $p$ with a `developing map'), that we may take $p$ to be a local isometry --- equivalently, we may assume that $p^{-1}(F)$ is a unit square for each face $F$ of $G'$. This in turn implies that $H = p^{-1}(G')$ is a tesselation of $\mathbb{R}^2$ with unit squares, and it follows easily that $H = \mathbb{L}^2$. Hence, $p|_{V(\mathbb{L}^2)}$ is a normal covering map (in the graph sense) from $\mathbb{L}^2$ to $H$.
\end{proof}

\begin{remark}
This proof can very easily be adapted to prove a weakening of Theorem \ref{thm:cover-d}, namely that if $G$ is a connected graph which is weakly $d$-locally $\mathbb{L}^d$, then $G$ is normally covered by $\mathbb{L}^d$. Indeed, one produces from $G$ a $d$-dimensional cubical CW complex $S$ in the natural way; $S$ is a complete, flat $d$-dimensional manifold so has universal cover $\mathbb{R}^d$.
\end{remark}

For our proof of Theorem \ref{thm:cover-d}, we will need the following.

\begin{observation}
\label{observation:opposite}
Let $d \geq 2$ be an integer. Let $G$ be a graph which is weakly $2$-locally $\mathbb{L}^d$. Then $G$ is $(2d)$-regular, and for each vertex $u \in V(G)$, the set of neighbours of $u$ can be partitioned into $d$ pairs
\begin{equation}\label{eq:d-pairs}\{\{a^{(i)}_1,a^{(i)}_{2}\}:\ i \in [d]\}\end{equation}
such that for all $i \in [d]$, $a_1^{(i)}$ and $a_2^{(i)}$ have $u$ as their only common neighbour, and for each pair $1 \leq i < j \leq d$ and each $(k,l) \in [2]^2$, there are exactly two common neighbours of $a^{(i)}_k$ and $a^{(j)}_l$, namely $u$ and one other vertex $c_{\{i,j\},(k,l)}$, with the $4{d \choose 2}$ vertices
$$(c_{\{i,j\},(k,l)}:\ \{i,j\} \in [n]^{(2)},\ (k,l) \in [2]^2)$$
being distinct. Further, the partition (\ref{eq:d-pairs}) is the unique partition with these properties. For each $i \in [d]$, we say that the two vertices $a^{(i)}_1$ and $a^{(i)}_2$ are `opposite one another across $u$'.
\end{observation}

\begin{remark}
\label{remark:3locally}
Suppose that $G$ is weakly 2-locally $\mathbb{L}^d$ for some integer $d \geq 2$. Suppose $a_1$ and $a_2$ are opposite one another across $u$, and $b \in \Gamma(u) \setminus \{a_1,a_2\}$. Then $a_1$ and $b$ have exactly two common neighbours, $u$ and $c_1$ (say). Similarly, $a_2$ and $b$ have exactly two common neighbours, $u$ and $c_2$ (say). If, in addition, $G$ is weakly $3$-locally $\mathbb{L}^d$, then $c_1$ and $c_2$ must be opposite one another across $b$, since $b$ is their only common neighbour. (Example \ref{example:Ld} shows that for each $d \geq 3$, this need not be true if $G$ is merely $2$-locally-$\mathbb{L}^d$.)
\end{remark}

\begin{lemma}
\label{lemma:condition}
Let $d \geq 2$ be an integer. Suppose $G$ is weakly $2$-locally $\mathbb{L}^d$. Suppose further that $p:V(\mathbb{L}^d) \to V(G)$ is a covering map from $\mathbb{L}^d$ to $G$. Then for every $x \in V(\mathbb{L}^d)$ and every $i \in [d]$, $p(x+e_i)$ is opposite $p(x-e_i)$ across $p(x)$.
\end{lemma}
\begin{proof}
Without loss of generality, we may assume that $x=0$, and $i=1$. Suppose for a contradiction that $p(e_1)$ is not opposite $p(-e_1)$ across $p(0)$. Then there exists $s \in \{\pm 1\}$ and $j \neq 1$ such that $p(se_j)$ is opposite $p(e_1)$ across $p(0)$, i.e. $p(0)$ is the unique common neighbour of $p(e_1)$ and $p(se_j)$. Since $p$ is a graph homomorphism, and $e_1+se_j$ is a common neighbour of $e_1$ and $se_j$, we must have $p(se_j+e_1) = p(0)$. But this contradicts the fact that $p$ is injective on $N(e_1)$, proving the lemma.
\end{proof}

We can now prove Theorem \ref{thm:cover-d}.

\begin{proof}[Proof of Theorem \ref{thm:cover-d}.]
We will construct a normal covering map $p:V(\mathbb{L}^d) \to V(G)$, directly. Choose any $v_0 \in V(G)$, and define $p(0)=v_0$. Write $\Gamma(v_0) = \{w_{\pm i}:\ i \in [d]\}$, where $w_{i}$ is opposite $w_{-i}$ across $v_0$, for each $i \in [d]$. Define $p(e_i) = w_{i}$ and $p(-e_i) = w_{-i}$ for each $i \in [d]$. We will show that there is a unique way of extending $p$ to all of $\mathbb{Z}^d$, in such a way that $p$ is a covering map from $\mathbb{L}^d$ to $G$.

We first make the following.
\begin{claim}
\label{claim:at-most-one}
There is at most one covering map $p$ from $\mathbb{L}^d$ to $G$ such that $p(0)=v_0$, $p(e_i) = w_i$ for all $i \in [d]$, and $p(-e_i) = w_{-i}$ for all $i \in [d]$.
\end{claim}
\begin{proof}[Proof of Claim \ref{claim:at-most-one}]
Note that if $p:\mathbb{Z}^d \to V(G)$ is a covering map from $\mathbb{L}^d$ to $G$, and $xyzwx$ is a 4-cycle in $\mathbb{L}^d$, then $p(x)p(y)p(z)p(w)p(x)$ is a 4-cycle in $G$. Since any three vertices of $G$ are in at most one 4-cycle, this implies that for any $4$-cycle $xyzwx$ of $\mathbb{L}^d$, the values of $p(x)$, $p(y)$ and $p(z)$ determine the value of $p(w)$.

This fact, combined with Lemma \ref{lemma:condition}, immediately implies that there is at most one extension of $p$ to all of $\mathbb{Z}^d$, such that $p$ is a covering map from $\mathbb{L}^d$ to $G$. (Starting with the set $N(0) = \{0\} \cup \{\pm e_i:\ i \in [d]\}$ on which $p$ is initially defined, we may cover all of $\mathbb{Z}^d$ by successive applications of the two operations of adding the third of three consecutive vertices on a straight path of length 2, and adding the fourth vertex of a 4-cycle.)  
\end{proof}

We now turn to showing existence. Let $x \in \mathbb{Z}^d$; we define $p(x)$ as follows. Let $P$ be any geodesic from $0$ to $x$ in $\mathbb{L}^d$. Let $H = \mathbb{L}^d[N(P)]$ be the subgraph of $\mathbb{L}^d$ induced on $N(P)$. We make the following.

\begin{claim}
\label{claim:exists-unique}
There exists a unique map $q:V(H)\to V(G)$ such that:
\begin{itemize}
\item[(1)] $q$ agrees with $p$ on $N(0)$, i.e. $q(0)=v_0$, $q(e_i) = w_i$, $q(-e_i) = w_{-i}$ for all $i \in [d]$;
\item[(2)] for each $y \in V(P)$ and each $i \in [d]$, $q(y+e_i)$ is opposite $q(y-e_i)$ across $y$;
\item[(3)] $q|\Gamma(y)$ is a bijection from $\Gamma(y)$ to $\Gamma(q(y))$, for all $y \in V(P)$;
\item[(4)] If $abcda$ is a 4-cycle containing an edge of $P$, then $q(a)q(b)q(c)q(d)q(a)$ is a 4-cycle in $G$.
\end{itemize}
\end{claim}
\begin{proof}[Proof of Claim \ref{claim:exists-unique}:]
WLOG WMA $x$ is in the positive orthant. Write $P = x_0 x_1 \ldots x_l$, where $x_0=0$ and $x_l = x$. For each $k \leq l$, let $P_k = x_0 x_1 \ldots x_k$. We construct $q$ (and show uniqueness) recursively. Let $k \leq l$, and suppose we have already defined $q$ on $N(P_k)$ such that $q$ satisfies properties (1)-(4) above when $P$ is replaced by $P_k$. We split into two cases.

{\em Case 1:} $x_{k-1},x_k,x_{k+1}$ are colinear. In this case, WLOG WMA $x_{k+1}-x_k = x_{k} - x_{k-1} = e_1$. To satisfy property (2) when $y=x_{k+1}$ and $i=1$, we must define $q(x_{k+1}+e_1)$ to be the vertex of $G$ which is opposite $q(x_{k})$ across $q(x_{k+1})$. For each $i >1$, observe that $q(x_k+e_i)$ and $q(x_{k+1}) = q(x_k+e_1)$ are not opposite one another across $q(x_k)$, so they have exactly two common neighbours, $q(x_k)$ and another vertex, $v_i$ say. To satisfy property (4) at the 4-cycle with vertex-set $\{x_k, x_{k+1}, x_{k+1}+e_i,x_k+e_i\}$, we must define $q(x_{k+1}+e_i) = v_i$. Similarly, for each $i >1$, observe that $q(x_k-e_i)$ and $q(x_{k+1}) = q(x_k+e_1)$ are not opposite one another across $q(x_k)$, so they have exactly two common neighbours, $q(x_k)$ and another vertex, $v_i'$ say. To satisfy property (4) at the 4-cycle with vertex-set $\{x_k, x_{k+1}, x_{k+1}-e_i,x_k-e_i\}$, we must define $q(x_{k+1}-e_i) = v_i'$. By Remark \ref{remark:3locally}, for each $i > 1$, $q(x_{k+1}+e_i)$ and $q(x_{k+1}-e_i)$ are opposite one another across $q(x_{k+1})$, so property (2) holds when $y=x_{k+1}$ for each $i > 1$. Property (3) now holds when $y=x_{k+1}$, because it holds when $y=x_k$, and $G$ is weakly 2-locally-$\mathbb{L}^d$.

{\em Case 2:} $x_{k-1},x_k,x_{k+1}$ are not colinear. In this case, WLOG WMA $x_{k+1}-x_{k} = e_1$ and $x_{k} - x_{k-1} = e_2$. To satisfy property (2) when $y=x_{k+1}$ and $i=1$, we must define $q(x_{k+1}+e_1)$ to be the vertex of $G$ which is opposite $q(x_{k})$ across $q(x_{k+1})$. Observe that $q(x_k+e_2)$ and $q(x_{k+1}) = q(x_k+e_1)$ are not opposite one another across $q(x_k)$, so they have exactly two common neighbours, $q(x_k)$ and another vertex, $v$ say. To satisfy property (4) at the 4-cycle with vertex-set $\{x_k, x_{k+1}, x_{k+1}+e_2,x_k+e_2\}$, we must define $q(x_{k+1}+e_2)=v$. Note that we have already defined $q(x_{k+1}-e_2) = q(x_{k-1}+e_1)$; this vertex is opposite $q(x_{k+1}+e_2)$ across $q(x_{k+1})$ by Remark \ref{remark:3locally}, so property (2) holds when $y=x_{k+1}$ and $i=2$. For each $i >2$, $q(x_{k+1}+e_i)$ and $q(x_{k+1}-e_i)$ must be defined exactly as in Case 1. As in Case 1, for each $i > 2$, $q(x_{k+1}+e_i)$ and $q(x_{k+1}-e_i)$ are then opposite one another across $q(x_{k+1})$, so property (2) holds when $y=x_{k+1}$ for each $i > 2$. Again as in Case 1, property (3) now holds when $y=x_{k+1}$, because it holds when $y=x_k$, and $G$ is weakly 2-locally-$\mathbb{L}^d$.
\end{proof}

We define $p(x)=q(x)$. Our next aim is to prove:
\begin{equation} \label{eq:well-defined-d} q(x) \textrm{ is independent of the choice of geodesic }P \textrm{ from }0 \textrm{ to } x.\end{equation}

To prove (\ref{eq:well-defined-d}), WLOG WMA $x$ is in the positive quadrant. Observe that for any two geodesics $P,P'$ from $0$ to $x$, we can get from $P$ to $P'$ by a sequence of `elementary switches', meaning operations which replace some sub-path $(y,y+e_i,y+e_i+e_j)$ by the sub-path $(y,y+e_j,y+e_i+e_j)$, for some $i \neq j$. Hence, it suffices to show that if $P$ and $P'$ differ from one another by an elementary switch, then the corresponding maps $q$ and $q'$ satisfy $q(x) = q'(x)$. Suppose then that $P = (x_0,x_1,\ldots,x_k,x_{k}+e_i,x_{k}+e_i+e_j,x_{k+3},\ldots, x_l)$ and $P' = (x_0,x_1,\ldots,x_k,x_{k}+e_j,x_{k}+e_i+e_j,x_{k+3},\ldots,x_l)$, where $x_0 = 0$ and $x_l=x$. WLOG, WMA $i=1$ and $j=2$, and let us write $y: = x_k$, so that
\begin{align*} 
P & = (x_0,x_1,\ldots,x_{k-1},y,y+e_1,y+e_1+e_2,x_{k+3},\ldots, x_l),\\
P' &= (x_0,x_1,\ldots,x_{k-1},y,y+e_2,y+e_1+e_2,x_{k+3},\ldots,x_l).
\end{align*}

We have $q(y) = q'(y)$ and $q(y+e_i) = q'(y+e_i)$ for all $i \in [d]$, since $P'_k = P_k$ and $q$ is uniquely determined on $N(P_k)$ (and $q'$ on $N(P'_k)$). In other words, $q$ and $q'$ agree with one another on $N(y)$. We assert that $q$ and $q'$ agree with one another on $N(y+e_1+e_2)$. Indeed, since $G$ is weakly 3-locally $\mathbb{L}^d$, there exists a graph isomorphism $\phi:\Link_3^{-}(y,\mathbb{L}^d) \to\Link_3^{-}(q(y),G)$ such that $\phi$ agrees with $q$ (and $q'$) on $N(y)$; it is easy to see that this $\phi$ is unique. Moreover, since $q$ and $q'$ satisfy properties (2) and (4), they must agree with $\phi$ on $B_{3}(y,\mathbb{L}^d)$, wherever they are defined. Since $N(y+e_1+e_2) \subset B_{3}(y,\mathbb{L}^d)$, it follows that $q$ and $q'$ must agree with one another on $N(y+e_1+e_2)$, as asserted. By the uniqueness part of Claim \ref{claim:exists-unique} (applied with $x_{k+2} = y+e_1+e_2$ in place of $0$), it follows now that $q = q'$ on $\{x_{k+2},x_{k+3},\ldots,x_l\}$, so in particular $q(x) = q'(x)$, proving (\ref{eq:well-defined-d}).

It follows from property (3) of Claim \ref{claim:exists-unique} that $p:\mathbb{Z}^d \to V(G)$ is a covering map from $\mathbb{Z}^d$ to $G$. Normality follows from Claim \ref{claim:at-most-one}. To see this, it suffices to show that the group of covering transformations of $p$ is transitive on the fibre $p^{-1}(0)$. So let $x \in \mathbb{Z}^d \setminus \{0\}$ with $p(x) = p(0)$. By Lemma \ref{lemma:condition}, $p(e_i)$ and $p(-e_i)$ must be opposite one another across $p(0)$ for all $i \in [d]$, and $p(x+e_i)$ and $p(x-e_i)$ must be opposite one another across $p(x)=p(0)$ for all $i \in [d]$. Hence, we have
$$\{\{p(x+e_i),p(x-e_i)\}:\ i \in [d]\} = \{\{p(e_i),p(-e_i)\}:\ i \in [d]\}.$$
Since $\Aut(\mathbb{L}^d) = T(\mathbb{Z}^d) \rtimes B_d$, and the elements of $B_d$ correspond precisely to the permutations of $\{\pm e_i:\ i \in [d]\}$ which preserve the partition
$$\{\{e_i,-e_i\}:\ i \in [d]\}$$
(see Fact \ref{fact:aut-ld}), we may choose $\alpha \in \Aut(\mathbb{L}^d)$ such that $\alpha(0) = x$ and $p(w) = p(\alpha(w))$ for all $w \in \Gamma(0)$. (Take $\alpha = g \circ t_x$, where $t_x$ is translation by $x$ and $g$ is the appropriate element of $B_d$.)

Observe that $p \circ \alpha$ is a covering map from $\mathbb{L}^d$ to $G$ which agrees with $p$ on $N(0)$. Hence, by Claim \ref{claim:at-most-one}, we have $p \circ \alpha = p$, so $\alpha \in \CT(p)$. Therefore, $\CT(p)$ acts transitively on $p^{-1}(0)$, as required. Hence, $p$ is a normal covering map from $\mathbb{L}^d$ to $G$. This completes the proof of Theorem \ref{thm:cover-d}.
\end{proof}

\begin{example}
\label{example:Ld}
We now give an example showing that Theorem \ref{thm:cover-d} is best possible, in the sense that for every integer $d \geq 3$, there exists a finite, connected graph which is $2$-locally-$\mathbb{L}^d$ but which is not covered by $\mathbb{L}^d$. We first deal with the case $d=3$.

Let us recall some more group-theoretic notions. If $\Gamma$ is a group, and $S \subset \Gamma$ with $\Id \notin S$ and $S^{-1}=S$, the {\em (right) Cayley graph of $G$ with respect to $S$} is the graph with vertex-set $G$ and edge-set
$$\{\{g,gs\}:\ g \in \Gamma,\ s \in S\}.$$
It is sometimes denoted by $\Cay(\Gamma,S)$.

We write finitely presented groups in the form
$$\langle a_1,a_2,\ldots,a_s;\ R_1, \ldots, R_N \rangle$$
where $a_1,\ldots,a_s$ are the {\em generators} and $R_1,\ldots,R_N$ are the {\em relations} (i.e., $R_i$ is an equation of the form $w_i=w_i'$, where $w_i$ and $w_i'$ are words in $a_1,\ldots,a_s$ and their inverses).

If $\Gamma$ is a finitely presented group with generators $a_1,\ldots,a_s$, the {\em length} of the word
$$a_{i_1}^{r_1} a_{i_2}^{r_2} \ldots a_{i_t}^{r_t}$$
is defined to be $\sum_{i=1}^{t}|r_i|$; for example, $a_1^{-2} a_2^2 a_2^{-1}$ has length 5. A {\em relator} is a word which evaluates to the identity in $\Gamma$. A relator is {\em trivial} if it evaluates to the identity in the free group with generators $a_1,a_2,\ldots,a_s$. For example, the trivial relators of length two are $a_i a_i^{-1}$ and $a_i^{-1}a_i$ (for $i \in [s]$).

Let $\Gamma$ be the finitely presented group with three generators defined by
\begin{equation}
\label{eq:pres}
\Gamma = \langle a,b,c ;\ a^{-1} b = c^2, b^{-1}c = a^2, c^{-1}a = b^2\rangle,
\end{equation}
and let $G = \Cay(\Gamma,\{a,b,c,a^{-1},b^{-1},c^{-1}\})$, i.e. $G$ is the graph of the finitely presented group $\Gamma$ with respect to the generators $a,b,c$. It can be checked (e.g. using a computer algebra system) that $\Gamma$ is a finite group, and in fact that $\Gamma \cong \mathbb{F}_2^4 \rtimes C_7$, so $|\Gamma|=112$. A concrete realisation of $\Gamma$ is the group
$$ T(\mathbb{F}_2^4) \rtimes \langle M \rangle \leq \Aff(\mathbb{F}_2^4),$$
where $\Aff(\mathbb{F}_2^4)$ denotes the group of all affine transformations of the vector space $\mathbb{F}_2^4$, i.e.
$$\Aff(\mathbb{F}_2^4) = \{(x \mapsto Ax+v):\ A \in \GL(\mathbb{F}_2^4),\ v \in \mathbb{F}_2^4\},$$
$$T(\mathbb{F}_2^4) = \{(x \mapsto x+v):\ v \in \mathbb{F}_2^4\} \leq \Aff(\mathbb{F}_2^4)$$
denotes the subgroup of all translations, and
$$M = \left( \begin{array}{cccc}0 & 0 & 1 & 0\\ 1 & 1 & 0 & 0 \\ 0 & 1 & 1 & 0\\ 0 & 0 & 0 & 1\end{array} \right).$$
Note that $M^7 = \Id$, so $\langle M \rangle \leq \GL(\mathbb{F}_2^4)$ is a cyclic group of order 7. We may take the generators $a,b,c$ to be
\begin{align*} a & = (x \mapsto Mx + (1,0,0,1)^{\top}),\\
b & = (x \mapsto M^2x + (1,0,1,1)^{\top}),\\
c & = (x \mapsto M^4 x + (0,1,0,1)^{\top}),\end{align*}
where we compose these functions from left to right, so that $(a \cdot b)(x) = b(a(x))$.

We now find the Abelianisation of $\Gamma$, i.e. the quotient group $\Gamma/[\Gamma,\Gamma]$, where
$$[\Gamma,\Gamma] : = \{ghg^{-1}h^{-1}:\ g,h \in \Gamma\}$$
denotes the commutator subgroup of $\Gamma$. Let $\bar{a},\bar{b},\bar{c}$ denote the images of $a,b,c$ in the Abelianization of $\Gamma$. Then, from the second relation, we have $\bar{c} = \bar{a}^2 \bar{b}$; substituting this into the first relation gives $\bar{b} = \bar{a} \bar{c}^2 = \bar{a} \bar{a}^4 \bar{b}^2$, so $\bar{b} = \bar{a}^{-5}$. Substituting this back into the second relation gives $\bar{c} = \bar{a}^2 \bar{b} = \bar{a}^{-3}$. The third relation then gives $1=\bar{c}^{-1} \bar{a} \bar{b}^{-2} = \bar{a}^3 \bar{a} \bar{a}^{10} = \bar{a}^{14}$. Hence, 
$$\Gamma/[\Gamma,\Gamma] = \langle \bar{a} \rangle \cong C_{14},$$
a cyclic group of order 14, and we have
$$\bar{b} = \bar{a}^9, \quad \bar{c} = \bar{a}^{11}, \quad \bar{a}^{-1} = \bar{a}^{13},\quad \bar{b}^{-1} = \bar{a}^5,\quad \bar{c}^{-1} = \bar{a}^3.$$
Since each of $\bar{a},\bar{b},\bar{c}$ is an odd power of $\bar{a}$, there is no relator of odd length (in $\bar{a},\bar{b},\bar{c}$ and their inverses) in $\Gamma/[\Gamma,\Gamma]$, and so there is no relator of odd length in the group $\Gamma$. Hence, $G$ has no odd cycle. In particular, $G$ has no cycle of length 3 or 5. Moreover, $\bar{a},\bar{b},\bar{c},\bar{a}^{-1},\bar{b}^{-1},\bar{c}^{-1}$ are all distinct elements of $\Gamma/[\Gamma,\Gamma]$. Hence, $a,b,c,a^{-1},b^{-1},c^{-1}$ are all distinct elements of $\Gamma$, so $G$ is $6$-regular. 

It can be checked that for the group $\Gamma$, the only relators of length 4 are the 24 relators arising from rearranging the relations in (\ref{eq:pres}) and taking inverses. (We suppress the details of this calculation, as it is straightforward but somewhat long; it can easily be done using a computer algebra system.) Hence, the following words of length2 appear in no non-trivial relator of length 4:
\begin{equation} \label{eq:opposite} ab,bc,ca,b^{-1}a^{-1},c^{-1}b^{-1},a^{-1}c^{-1} \end{equation}
all the other non-trivial words of length two appear as the initial two letters of exactly one non-trivial relator of length 4.

It follows that $G$ is $2$-locally-$\mathbb{L}^3$. Indeed, since $G$ is vertex-transitive, it suffices to check this at the vertex $\Id$, only. In other words, we must construct a map $\psi:B_2(0,\mathbb{L}^3) \to V(G)$ which is an isomorphism from $\Link_2(0,\mathbb{L}^2)$ to $\Link_2(\textrm{Id},G)$, with $\psi(0)=\textrm{Id}$.

Let $S : = \{a,b,c,a^{-1},b^{-1},c^{-1}\}$. Let us say that two distinct elements $x,y \in S$ are {\em complementary} if $x^{-1}y$ appears in no non-trivial relator of length 4. (Note that this relation is symmetric, as $x^{-1}y$ appears in the list (\ref{eq:opposite}) if and only if $y^{-1}x$ does. Moreover, each element of $S$ is complementary to exactly one other element of $S$.) We can construct an appropriate map $\psi$ by choosing the six images $\psi(\pm e_i)$ to be distinct elements of $S$, in such a way that $\psi(e_i)$ and $\psi(-e_i)$ are complementary for each $i \in \{1,2,3\}$. For $i \neq j$ and $s,t \in \{\pm 1\}$, we can then define $\psi(se_i + te_j)$ as follows. Let $x = \psi(se_i)$ and $y = \psi(te_j)$. Then $x$ and $y$ are distinct and not complementary, so $x^{-1}y$ appears as the initial two letters of exactly one non-trivial relator of length 4, say $x^{-1} y u v = \Id$. Define $\psi(se_i + te_j) = yu\ (= xv^{-1})$. Finally, for each $i$ and each $s \in \{\pm 1\}$, define $\psi(2se_i)$ as follows. Let $x = \psi(se_i)$. Let $y$ be the unique element of $S$ such that $x^{-1}$ and $y$ are complementary, and define $\psi(2se_i) = xy$.

An explicit choice of $\psi$ is as follows.
\begin{equation*}
\begin{split}
\psi(0) & = \textrm{Id},\\
\psi(e_1) & = a,\\
\psi(-e_1) & = c^{-1},\\
\psi(e_2) & = b,\\
\psi(-e_2) & = a^{-1},\\
\psi(e_3) & = c,\\
\psi(-e_3) & = b^{-1},\\
\psi(2e_1) & = ab,\\
\psi(-2e_1)&=c^{-1}b^{-1},\\
\psi(2e_2)&= bc,\\
\psi(-2e_2)&=a^{-1}c^{-1},\\
\psi(2e_3) &= ca,\\
\psi(-2e_3) &= b^{-1}a^{-1},
\end{split}
\quad
\begin{split}
\psi(e_1+e_2) & = ac = bc^{-1},\\
\psi(e_1-e_2) & = ac^{-1} = a^{-1}b^{-1},\\
\psi(-e_1+e_2) &= c^{-1}a = b^2,\\
\psi(-e_1-e_2) &= c^{-1}b = a^{-2},\\
\psi(e_1+e_3) &= ab^{-1} = cb,\\
\psi(e_1-e_3)&= a^2 = b^{-1}c,\\
\psi(-e_1+e_3)& = c^{-1}a^{-1} = cb^{-1},\\
\psi(-e_1-e_3) &= c^{-2} = b^{-1}a,\\
\psi(e_2+e_3) &= ca^{-1}=ba,\\
\psi(e_2-e_3)&= ba^{-1} = b^{-1}c^{-1},\\
\psi(-e_2+e_3)& = a^{-1}b = c^2,\\
\psi(-e_2-e_3) &= a^{-1}c = b^{-2}.
\end{split}
\end{equation*}

\newpage
\begin{figure}[here]
 \centering
 \includegraphics[width = 0.8\textwidth]{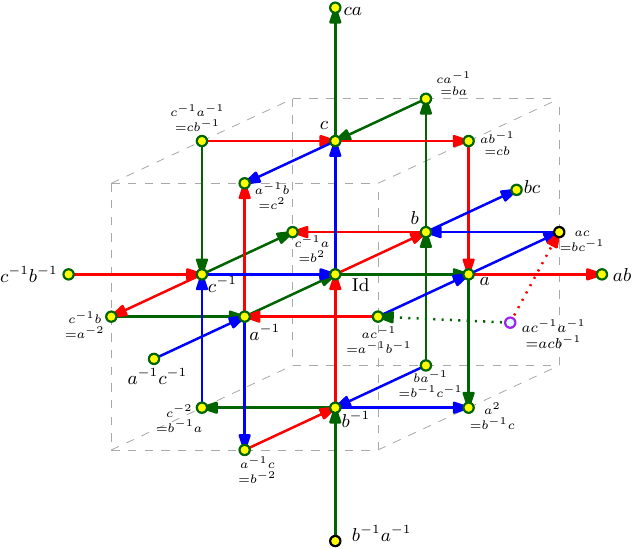}
 \label{fig:Y}
  \caption{The map $\psi$.}
\end{figure}

Using the facts that $G$ has no 3-cycle or 5-cycle, together with (\ref{eq:opposite}), it is easy to see that $\psi$ is an isomorphism from $\Link_2(0,\mathbb{L}^2)$ to $\Link_2(\textrm{Id},G)$. We may conclude that $G$ is 2-locally-$\mathbb{L}^3$.

On the other hand, we claim that $G$ is not covered by $\mathbb{L}^3$. Indeed, suppose for a contradiction that $p:\mathbb{L}^3 \to G$ is a cover map. By considering $p \circ \phi$ for some $\phi \in \Aut(\mathbb{L}^3)$ if necessary, we may assume that $p(0) = \textrm{Id}$ and $p(e_1)=a$. By Lemma \ref{lemma:condition}, $p(-e_1)$ must be opposite $p(e_1)=a$ across $p(0)=\textrm{Id}$, so $p(-e_1) = c^{-1}$. By considering $p \circ \phi$ for some $\phi \in \Aut(\mathbb{L}^3)$ fixing the $x$-axis, if necessary, we may assume that $p(e_2) = b$. Then, by Lemma \ref{lemma:condition}, $p(-e_2)$ must be opposite $p(e_2)=b$ across $p(0)=\textrm{Id}$, so $p(-e_2) = a^{-1}$. 

The only common neighbours of $p(e_1)=a$ and $p(e_2)=b$ are $\textrm{Id}$ and $ac=bc^{-1}$. Hence, we must have $p(e_1+e_2) = ac=bc^{-1}$ (clearly, $p(e_1 + e_2) \neq \textrm{Id}$, as $p$ must be bijective on $\Gamma(e_1)$). Similarly, the only common neighbours of $p(e_1)=a$ and $p(-e_2)=a^{-1}$ are $\textrm{Id}$ and $ac^{-1}=a^{-1}b^{-1}$, so we must have $p(e_1-e_2) = ac^{-1}=a^{-1}b^{-1}$. But then $p(e_1+e_2)$ and $p(e_1-e_2)$ have two common neighbours, namely $a$ and $ac^{-1}a^{-1} = acb^{-1}$, so they are not opposite one another, contradicting Lemma \ref{lemma:condition}. Hence, $G$ is not covered by $\mathbb{L}^3$, as claimed.

For $d \geq 4$, we let $\Gamma_d = \Gamma \times \mathbb{Z}_{14}^{d-3}$, and $G_d = \Cay(\Gamma_d,S_d)$ where
\begin{align*} S_d & = \{(a,0),(b,0),(c,0),(a^{-1},0),(b^{-1},0),(c^{-1},0)\}\\
& \cup \{(\textrm{Id},f_i):\ i \in [d-3]\} \cup \{(\textrm{Id},-f_i):\ i \in [d-3]\}),\end{align*}
and where $f_i = (0,0,\ldots,0,1,0,\ldots,0) \in \mathbb{Z}_{14}^{d-3}$ denotes the $i$th unit vector in $\mathbb{Z}_{14}^{d-3}$. It is easy to see (using the $d=3$ case) that $G_d$ is $2$-locally $\mathbb{L}^d$, but is not covered by $\mathbb{L}^d$.
\end{example}

\section{Proofs of `algebraic' structure theorems}  
In this section, we use standard results and techniques from topological graph theory and group theory, combined with the `topological' structure theorems of the previous two sections, to deduce Corollaries \ref{corr:torusklein2} and \ref{corr:structure-d}, which concern the `algebraic' (quotient-type) structure of graphs which have the 4-cycle-wheel property, or which are weakly 3-locally $\mathbb{L}^d$ (for $d \geq 3$).

\subsubsection*{The $d=2$ case.}

We first prove the following.

\begin{proposition}
\label{prop:quotient-wheel}
If $G$ is a finite, connected graph with the 4-cycle wheel property, then $G$ is isomorphic to $\mathbb{L}^2/\Gamma$, where $\Gamma$ is a subgroup of $\Aut(\mathbb{L}^2)$ with $|\mathbb{Z}^2/\Gamma| < \infty$ and with minimum displacement at least 3.
\end{proposition}
\begin{proof}
Let $G$ be a finite, connected graph with the 4-cycle wheel property. It follows immediately from Theorem \ref{thm:cover-2}, Lemma \ref{lemma:free} and Lemma \ref{lemma:iso} that $G$ is isomorphic to $\mathbb{L}^2/\Gamma$, where $\Gamma$ is a subgroup of $\Aut(\mathbb{L}^2)$ which acts freely on $\mathbb{L}^2$. Clearly, we have $|\mathbb{Z}^2/\Gamma| = |V(G)| < \infty$. Observe that if $\Gamma$ has minimum displacement at most $2$, then $\mathbb{L}^2/\Gamma$ is not 4-regular, so it does not have the 4-cycle wheel property. Hence, $\Gamma$ has minimum displacement at least 3.
\end{proof}

Our next step is to classify the subgroups $\Gamma$ of $\Aut(\mathbb{L}^2)$ which have $|\mathbb{Z}^2/\Gamma| < \infty$ and minimum displacement at least 3. Using the description of $\Aut(\mathbb{L}^2)$ in Fact \ref{fact:aut-ld}, it is easy to check the following.
\begin{claim}
\label{claim:torsion-free2}
If $\Gamma \leq \Aut(\mathbb{L}^2)$ has minimum displacement at least 3, then $\Gamma$ is torsion-free. \qed
\end{claim}

We also need the $d=2$ case of the following simple fact (we will need the $d \geq 3$ case later).
\begin{claim}
\label{claim:rank-d}
If $\Gamma \leq \Aut(\mathbb{L}^d)$ with $|\mathbb{Z}^d/\Gamma| < \infty$, then the lattice of translations of $\Gamma$ has rank $d$.
\end{claim}
\begin{proof}
Let $L_{\Gamma}$ denote the lattice of translations of $\Gamma$. Suppose $\rank(L_{\Gamma}) < d$. Let $\{v_1,\ldots,v_r\}$ be a $\mathbb{Z}$-basis for $L_{\Gamma}$; then $r < d$. Choose $w \in \mathbb{Z}^d \setminus \langle v_1,\ldots,v_r \rangle_{\mathbb{R}}$, where $\langle v_1,\ldots,v_r \rangle_{\mathbb{R}}$ denotes the subspace of $\mathbb{R}^d$ spanned by $v_1,\ldots,v_r$ over $\mathbb{R}$. We assert that for any $x \in \mathbb{Z}^d$, there are at most $2^d d!$ elements of $\{x+\lambda w:\ \lambda \in \mathbb{Z}\}:=L$ in the same $\Gamma$-orbit as $x$. Indeed, suppose otherwise. Let $S = \{\gamma \in \Gamma:\ \gamma(x) \in \{x+\lambda w:\ \lambda \in \mathbb{Z}\}\}$; then $|S| > 2^d d!$. Since $\Gamma \leq \Aut(\mathbb{L}^d) = T(\mathbb{Z}^d) \rtimes B_d$ and $|B_d|=2^d d!$ (see Fact \ref{fact:aut-ld}), by the pigeonhole principle, there exist $g \in B_d$ and two distinct translations $t_1,t_2 \in T(\mathbb{Z}^d)$ such that $t_1g, t_2g \in S$. Notice that $(t_1 g) (t_2 g)^{-1} = t_1 t_2^{-1}$ is a translation in $\Gamma \setminus \{\Id\}$. But there exist $y,z \in L$ such that $t_1 g(x)=y$ and $t_2 g(x)=z$, so $t_1 t_2^{-1}(z)=y$, so $t_1 t_2^{-1}$ fixes the set $L$, so $t_1t_2^{-1}$ is a translation by $ \mu w$ for some $\mu \in \mathbb{Z} \setminus \{0\}$, so $w \in \langle v_1,v_2,\ldots,v_r \rangle_{\mathbb{R}}$, contradicting our choice of $w$, and proving our assertion. The assertion implies that $\{\lambda w: \lambda \in \mathbb{Z}\}$ meets infinitely many $\Gamma$-orbits. Hence, $|\mathbb{Z}^d/\Gamma| = \infty$, proving the claim.
\end{proof}

Recall that we can view $\Aut(\mathbb{L}^2)$ as a discrete subgroup of $\Isom(\mathbb{R}^2)$ (see Facts \ref{fact:embedding-isometries} and \ref{fact:discrete}), and of course the same holds for any subgroup of $\Aut(\mathbb{L}^2)$. If $\Gamma \leq \Aut(\mathbb{L}^2)$ with $|\mathbb{Z}^2/\Gamma| < \infty$, then by Claim \ref{claim:rank-d}, the lattice of translations of $\Gamma$ has rank $2$, so $\Gamma$ is a 2-dimensional crystallographic group (see definition \ref{definition:crystallographic}). Combining this fact with Claim \ref{claim:torsion-free2} implies that if $\Gamma \leq \Aut(\mathbb{L}^2)$ with $|\mathbb{Z}^2/\Gamma| < \infty$ and with minimum displacement at least 3, then $\Gamma$ is a torsion-free 2-dimensional crystallographic group, i.e. a 2-dimensional Bieberbach group (see definition \ref{definition:bieberbach}). The classification of 2-dimensional Bieberbach groups (see for example \cite{charlap}) says the following.
\begin{proposition}
Let $\Gamma \leq \Isom(\mathbb{R}^2)$ be a 2-dimensional Bieberbach group. Then either
\begin{enumerate}
\item $\Gamma=\langle t_1,t_2\rangle$, where $t_1$ and $t_2$ are linearly independent translations, or
\item $\Gamma = \langle g,t \rangle$, where $g$ is a glide-reflection and $t$ is a translation in a direction perpendicular to the reflection-axis of $g$.
\end{enumerate}
\end{proposition}
In case (1), $\mathbb{R}^2/\Gamma$ is a (topological) torus; in case (2), $\mathbb{R}^2/\Gamma$ is a (topological) Klein bottle. This immediately implies Corollary \ref{corr:torusklein2}.

\subsubsection*{The $d \geq 3$ case.}
\begin{proof}[Proof of Corollary \ref{corr:structure-d}.]
Let $G$ be a finite, connected graph which is weakly 3-locally $\mathbb{L}^d$. It follows immediately from Theorem \ref{thm:cover-d}, Lemma \ref{lemma:free} and Lemma \ref{lemma:iso} that $G$ is isomorphic to $\mathbb{L}^d/\Gamma$, where $\Gamma$ is a subgroup of $\Aut(\mathbb{L}^d)$ which acts freely on $\mathbb{L}^d$.

As in the $d=2$ case, recall that $\Aut(\mathbb{L}^d)$ can be viewed as a discrete subgroup of $\Isom(\mathbb{R}^d)$, and of course the same holds for any subgroup of $\Aut(\mathbb{L}^d)$. Since $\Gamma \leq \Aut(\mathbb{L}^d)$ with $|\mathbb{Z}^d/\Gamma| = |V(G)| < \infty$, by Claim \ref{claim:rank-d}, the lattice of translations $L_{\Gamma}$ of $\Gamma$ has rank $d$, so $\mathbb{R}^d/\Gamma$ is compact, and by definition, $\Gamma$ is a $d$-dimensional crystallographic group. Hence, by Fact \ref{fact:proper-orbifold}, $\mathbb{R}^d/\Gamma$ can be given the structure of a $d$-dimensional topological orbifold.
\end{proof}

Our aim is now to obtain an exact (algebraic) characterisation of the graphs which are weakly $3$-locally $\mathbb{L}^d$, and to show that the orbit space $\mathbb{R}^d/\Gamma$ in Corollary \ref{corr:structure-d} need not be a topological manifold when $d \geq 7$. For this, we need more detailed information about the group $\Gamma$. We first prove the following easy lemma.

\begin{lemma}
\label{lemma:min-dist-r}
Let $d,r \in \mathbb{N}$, and let $\Gamma \leq \Aut(\mathbb{L}^d)$. Let $D(\Gamma)$ denote the minimum displacement of $\Gamma$. Then
\begin{itemize}
\item[(i)] $\mathbb{L}^d/\Gamma$ is $r$-locally $\mathbb{L}^d$ if and only if $D(\Gamma) \geq 2r+2$.
\item[(ii)] $\mathbb{L}^d/\Gamma$ is weakly $r$-locally $\mathbb{L}^d$ if and only if $D(\Gamma) \geq 2r+1$.
\end{itemize}
\end{lemma}
 
\begin{proof}
(i) Suppose $\Gamma \leq \Aut(\mathbb{L}^d)$ has $D(\Gamma) \geq 2r+2$. Let $x \in V(\mathbb{L}^d)$. Then it is easy to check that the orbit map $x \mapsto \Orb_{\Gamma}(x)$ is a graph isomorphism from $\Link_r(x,\mathbb{L}^d)$ to $\Link_{r}(\Orb(x),\mathbb{L}^d/\Gamma)$ which maps $x$ to $\Orb(x)$, so $\mathbb{L}^d/\Gamma$ is $r$-locally $\mathbb{L}^d$.

On the other hand, suppose $\Gamma \leq \Aut(\mathbb{L}^d)$ has $D(\Gamma) \leq 2r+1$. If $D(\Gamma) = 2r+1$, then there exist $y,z \in V(\mathbb{L}^d)$ and $\gamma \in \Gamma$ with $d_{\mathbb{L}^d}(y,z) = 2r+1$ and $\gamma(y)=z$, so $\Orb(y)=\Orb(z)$. Let $(y,x_1,x_2,\ldots,x_{2r},z)$ be a geodesic in $\mathbb{L}^d$ from $y$ to $z$; then
$$\Orb(y) \Orb(x_1) \Orb(x_2) \ldots \Orb(x_{2r})\Orb(y)$$
is a cycle in $\mathbb{L}^d/\Gamma$ of length $2r+1$, so $\mathbb{L}^d/\Gamma$ is not $r$-locally $\mathbb{L}^d$. We may assume henceforth that $D(\Gamma) \leq 2r$.

Now note that if $\{\Orb(x),\Orb(y)\} \in E(\mathbb{L}^d/\Gamma)$, then there exists $z \in \Orb(y)$ such that $\{x,z\} \in E(\mathbb{L}^d)$. Indeed, if $\{\Orb(x),\Orb(y)\} \in E(\mathbb{L}^d/\Gamma)$, then there exist $\gamma,\gamma' \in \Gamma$ such that $\{\gamma(x),\gamma'(y)\} \in E(\mathbb{L}^d)$, and therefore $\{x,\gamma^{-1}\gamma'(y)\} \in E(\mathbb{L}^d)$, so we may take $z = \gamma^{-1}\gamma'(y)$. Similarly, if
$$(\Orb(x_0),\Orb(x_1),\Orb(x_2),\ldots,\Orb(x_l))$$
is a path in $\mathbb{L}^d/\Gamma$, then there exist $z_i \in \Orb(x_i)$ for each $i \in [l]$ such that $(x_0,z_1,z_2,\ldots,z_l)$ is a path in $\mathbb{L}^d$. It follows that the number of vertices of $\mathbb{L}^d/\Gamma$ of distance at most $r$ from $\Orb(x)$ is precisely the number of $\Gamma$-orbits of $V(\mathbb{L}^d)$ intersecting the ball $B_r(x,\mathbb{L}^d)$.

Since $D(\Gamma) \leq 2r$, there exist $y,z \in V(\mathbb{L}^d)$ and $\gamma \in \Gamma$ such that $d_{\mathbb{L}^d}(y,z) \leq 2r$ and $\gamma(y)=z$, so $\Orb(y)=\Orb(z)$. Choose $x \in V(\mathbb{L}^d)$ such that $y,z \in B_r(x,\mathbb{L}^d)$. Then $\Orb(y)$ intersects $B_r(x,\mathbb{L}^d)$ in at least two vertices ($y$ and $z$), so
$$ |B_{r}(\Orb(x),\mathbb{L}^d/\Gamma)| = \textrm{no. of }\Gamma\textrm{-orbits of }V(\mathbb{L}^d)\textrm{ intersecting }B_r(x,\mathbb{L}^d) < |B_r(x,\mathbb{L}^d)|.$$
It follows that $\mathbb{L}^d/\Gamma$ is not $r$-locally $\mathbb{L}^d$.

(ii) Suppose $\Gamma \leq \Aut(\mathbb{L}^d)$ has $D(\Gamma) \geq 2r+1$. Let $x \in V(\mathbb{L}^d)$. It is easy to check that the orbit map $x \mapsto \Orb_{\Gamma}(x)$ is a bijective graph homomorphism from $\Link_r(x,\mathbb{L}^d)$ to $\Link_{r}(\Orb(x),\mathbb{L}^d/\Gamma)$ which maps $x$ to $\Orb(x)$. Hence, $\mathbb{L}^d/\Gamma$ is weakly $r$-locally $\mathbb{L}^d$.

On the other hand, suppose $\Gamma \leq \Aut(\mathbb{L}^d)$ has $D(\Gamma) \leq 2r$. Then, by the same argument as in part (i),
$$|B_{r}(\Orb(x),\mathbb{L}^d/\Gamma)| = \textrm{no. of }\Gamma\textrm{-orbits of }V(\mathbb{L}^d)\textrm{ intersecting }B_r(x,\mathbb{L}^d) < |B_r(x,\mathbb{L}^d)|,$$
so $\mathbb{L}^d/\Gamma$ is not weakly $r$-locally $\mathbb{L}^d$.
\end{proof}

We can now deduce the following exact characterisation of the graphs which are $3$-locally $\mathbb{L}^d$, or weakly $3$-locally $\mathbb{L}^d$.
\begin{proposition}
\label{prop:quotient-d}
Let $d \in \mathbb{N}$ with $d \geq 2$, and let $G$ be a connected graph. Then
\begin{enumerate}
\item[(i)] $G$ is 3-locally $\mathbb{L}^d$ if and only if $G$ is isomorphic to $\mathbb{L}^d/\Gamma$, where $\Gamma \leq \Aut(\mathbb{L}^d)$ with $D(\Gamma) \geq 8$;
\item[(ii)] $G$ is weakly 3-locally $\mathbb{L}^d$ if and only if $G$ is isomorphic to $\mathbb{L}^d/\Gamma$, where $\Gamma \leq \Aut(\mathbb{L}^d)$ with $D(\Gamma) \geq 7$.
\end{enumerate}
\end{proposition}
\begin{proof}
Let $G$ be a connected graph which is weakly 3-locally $\mathbb{L}^d$. It follows immediately from Theorem \ref{thm:cover-d}, Lemma \ref{lemma:free} and Lemma \ref{lemma:iso} that $G$ is isomorphic to $\mathbb{L}^d/\Gamma$, where $\Gamma$ is a subgroup of $\Aut(\mathbb{L}^d)$ which acts freely on $\mathbb{L}^d$. By part (ii) of Lemma \ref{lemma:min-dist-r} applied with $r=3$, we have $D(\Gamma) \geq 7$. If in addition, $G$ is 3-locally $\mathbb{L}^d$, then by part (i) of Lemma \ref{lemma:min-dist-r}, we have $D(\Gamma) \geq 8$. The converse of each statement follows immediately from Lemma \ref{lemma:min-dist-r}.
\end{proof}

To show that for each $d \geq 7$, the orbit space $\mathbb{R}^d/\Gamma$ in Corollary \ref{corr:structure-d} need not be a topological manifold, it suffices to exhibit (for each $d \geq 7$) a subgroup $\Gamma \leq \Aut(\mathbb{L}^d)$ satisfying the conditions of Proposition \ref{prop:quotient-d} (ii) and such that $\mathbb{R}^d/\Gamma$ is not a topological manifold. We do this below.

\begin{example}
\label{example:not-manifold}
Let
$$\Gamma = \langle \{(x \mapsto x+2de_i):\ i \in [d]\} \cup \{(x \mapsto (1,1,\ldots,1)-x)\}\rangle \leq \Aut(\mathbb{L}^d).$$
Note that $\Gamma$ contains an element of order 2, so is not torsion-free and therefore is not a Bieberbach group. It has $|\mathbb{Z}^d/\Gamma| < \infty$ and has minimum displacement $d$, so by Lemma \ref{lemma:min-dist-r}, $\mathbb{L}^d/\Gamma$ is weakly 3-locally $\mathbb{L}^d$ if $d \geq 7$ (and indeed 3-locally-$\mathbb{L}^d$, if $d \geq 8$). However, for each $d \geq 3$, $\mathbb{R}^d/\Gamma$ is not a topological manifold. This follows, for example, from the fact that a small metric ball around the point $[(1/2,1/2,\ldots,1/2)]$ has topological boundary homeomorphic to $(d-1)$-dimensional projective space $\mathbb{RP}^{d-1}$, whereas it is known that no subset of $\mathbb{R}^{d}$ is homeomorphic to $\mathbb{RP}^{d-1}$ if $d \geq 3$. (See \cite{hatcher} Chapter 3, p. 256 for a proof of this in the case of odd $d$, and \cite{thom} for a proof for all $d \geq 3$.)
\end{example}

\section{Conclusion and related problems}
\label{sec:conc}
Theorem \ref{thm:cover-d} states that a connected graph $G$ which is weakly 3-locally $\mathbb{L}^d$ is normally covered by $\mathbb{L}^d$. Example \ref{example:Ld} shows that for each $d \geq 3$, the hypothesis of Theorem \ref{thm:cover-d} cannot be weakened to $G$ being 2-locally $\mathbb{L}^d$. Nevertheless, Example \ref{example:Ld} is still `highly structured', and we are not aware of any essentially different alternative constructions. It would be interesting to obtain a (weaker) structure theorem for graphs which are $2$-locally $\mathbb{L}^d$ (for each $d \geq 3$), and to do the same for graphs which are weakly 2-locally $\mathbb{L}^d$ (for each $d \geq 3$).

Our results imply that if $r \geq r^{*}(d)$, then a connected graph which is $r$-locally $\mathbb{L}^d$ is covered by $\mathbb{L}^d$, where
$$r^{*}(d) := \begin{cases} 2 & \textrm{if } d = 2;\\ 3 & \textrm{if } d \geq 3.\end{cases}$$
 Benjamini and Georgakopoulos conjectured the following generalisation of this fact.

\begin{conjecture}[Benjamini, Georgakopoulos]
\label{conj:finitely-presented}
Let $\Gamma$ be a finitely presented group, and let $F$ be a connected, locally finite Cayley graph of $\Gamma$. Then there exists $r \in \mathbb{N}$ such that if $G$ is a graph which is $r$-locally $F$, then $F$ covers $G$.
\end{conjecture}
(Note that this conjecture appeared in \cite{coarse-geometry} without the assumption of $\Gamma$ being finitely presented. It is easy to see that the conjecture is false without this assumption, however, and the conjecture was afterwards amended to the above.)

In \cite{dST}, de la Salle and Tessera disprove Conjecture \ref{conj:finitely-presented}; they also prove several positive results, among which is the following. If $F$ is a graph and $k \in \mathbb{N}$ with $k \geq 2$, we let $P_k(F)$ denote the $2$-dimensional polygonal CW complex whose 1-skeleton is $F$, and whose $2$-cells are the cycles of length at most $k$ in $F$. Following \cite{dST}, we say that $F$ is {\em simply connected at level $k$} if $P_k(F)$ is simply connected, and we say that $F$ is {\em large-scale simply connected} if there exists $k \geq 2$ such that $F$ is simply connected at level $k$.
\begin{theorem}[De La Salle, Tessera]
\label{thm:positive}
Let $F$ be a connected, locally finite graph which is large-scale simply connected, and which has $|V(F)/\Aut(F)| < \infty$ and $|\Stab_{\Aut(F)}(v)| < \infty$ for all $v \in V(F)$. Then there exists $r = r(F) \in \mathbb{N}$ such that if $G$ is a graph which is $r$-locally $F$, then $F$ covers $G$.
\end{theorem}

Since $\mathbb{L}^d$ satisfies the hypotheses of Theorem \ref{thm:positive} for any $d \in \mathbb{N}$, the $\mathbb{L}^d$-case of Theorem \ref{thm:positive} implies a weakened version of Theorem \ref{thm:cover-d}.

It would be interesting to determine more precisely the class of graphs $F$ for which the conclusion of Conjecture \ref{conj:finitely-presented} holds; note that this class contains $T_d$, the infinite $d$-regular tree, which (for $d \geq 3$) does not satisfy the hypotheses of Theorem \ref{thm:positive}, as $\Aut(T_d)$ has infinite vertex-stabilizers. It would also be of interest to obtain good quantitative bounds on $r(F)$ for graphs $F$ in various classes (such as Cayley graphs on nilpotent groups of step $k$).

This paper deals with the properties of {\em all} finite graphs which are $r$-locally $\mathbb{L}^d$, or weakly $r$-locally $\mathbb{L}^d$, for various $r$. In a subsequent paper \cite{in-preparation}, we will study the {\em typical} properties of a uniform random $n$-vertex graph which is $r$-locally $\mathbb{L}^d$. It turns out that, for any integer $r \geq 2$, a graph chosen uniformly at random from the set of all $n$-vertex graphs which are $r$-locally $\mathbb{L}^2$, has largest component of order $o(n)$ and automorphism group of order at least $\exp(\Omega((\log n)^2))$, with high probability. Similarly, for any integer $d \geq 3$ and any integer $r \geq 3$, a graph chosen uniformly at random from the set of all $n$-vertex graphs which are $r$-locally $\mathbb{L}^d$, has largest component of order $o(n)$ and automorphism group of order at least $\exp(\Omega((\log n)^2))$, with high probability. This is in stark contrast to $G_{2d}(n)$, the random $(2d)$-regular graph on $n$ vertices, which is connected with high probability for all $d \geq 2$, and has trivial automorphism group with high probability for all $d \in \{2,3,\ldots,\lfloor (n-4)/2 \rfloor\}$. In \cite{in-preparation}, we make several conjectures regarding what happens when $\mathbb{L}^d$ is replaced by other Cayley graphs of finitely generated groups.

\subsection*{Acknowledgements}
We thank \'Eric Colin de Verdi\`ere, Tsachik Gelander, Agelos Georgakopoulos, Atsuhiro Nakamoto, Behrang Noohi, Oliver Riordan, Mikael de la Salle, Neil Strickland and Romain Tessera for very helpful discussions. We thank the editors of {\em Forum of Mathematics: Sigma} for suggesting a way of shortening our original proof of Theorem \ref{thm:cover-2}, and for other helpful comments. Finally, we thank an anonymous referee for very kindly providing Figure 1.

\end{document}